\documentclass[12pt]{amsart}
\usepackage{a4}
\usepackage{amssymb}
\usepackage{amsmath}
\usepackage{amsthm}
\usepackage{amstext}
\usepackage{amscd}
\usepackage{latexsym}
\usepackage{graphics}

\theoremstyle{plain}
\newtheorem{thm}{Theorem}[section]
\newtheorem{lemma}[thm]{Lemma}
\newtheorem{lem}[thm]{Lemma}
\newtheorem{cor}[thm]{Corollary}
\newtheorem{prop}[thm]{Proposition}
\theoremstyle{definition}

\renewcommand{\tilde}{\widetilde}
\newcommand{\sO}{{\mathcal O}}

\newcommand{\N}{{\mathbb N}}

\newcommand{\Z}{{\mathbb Z}}
\newcommand{\Gal}{{\rm Gal}}

\newcommand{\ilim}{\mathop{\varprojlim}\limits} % inverse limit
\input{xy}
\xyoption{all}
\begin{document}

\title[On Hesselholt's conjecture]{On Cohomology of Witt vectors of algebraic integers and a conjecture of Hesselholt} 

\author[A. Hogadi]{Amit Hogadi}

\address{School of Mathematics, Tata Institute of Fundamental
Research, Homi Bhabha Road, Bombay 400005, India}

\email{amit@math.tifr.res.in}

\author[S. Pisolkar]{Supriya Pisolkar}

\address{School of Mathematics, Tata Institute of Fundamental
Research, Homi Bhabha Road, Bombay 400005, India}

\email{supriya@math.tifr.res.in}

\date{}

\begin{abstract}

Let $K$ be a complete discrete valued field of characteristic zero with residue field $k_K$ of characteristic $p > 0$. Let $L/K$ be a finite Galois extension with the Galois group $G$ and suppose that the induced extension of residue fields $k_L/k_K$ is separable. In \cite{lh}, Hesselholt conjectured that $H^1(G,W(\sO_L))$ is zero, where $\sO_L$ is the ring of integers of $L$ and $W(\sO_L)$ is the Witt ring of $\sO_L$ w.r.t. the prime $p$. He partially proved this conjecture for a large class of extensions. In this paper, we prove Hesselholt's conjecture for all Galois extensions.  
\end{abstract}

\maketitle

\section{Introduction} \label{intro}

\noindent Let $p$ be a prime number and $K$ be a complete discrete valued field of characteristic zero with residue field $k_K$ of characteristic $p$. $L/K$ be a finite Galois extension of complete discrete valued fields as above with Galois group $G$. Suppose that $k_L/k_K$ is separable. In \cite{lh}, Hesselholt conjectured that the  proabelian group $\ilim H^1(G,W_n(\sO_L))$ vanishes, where $W_n(\sO_L)$ is the ring of Witt vectors of length $n$ in $\sO_L$. As explained in \cite{lh}, this can be viewed as an analogue of Hilbert theorem $90$ for the Witt ring $W(\sO_L)$.

In order to prove this conjecture one easily reduces to the case where $L/K$ is a totally ramified Galois extension of degree $p$ (see Lemma \ref{rd}). For such extension, let $s=s({\rm L/K})$ be the ramification break (see \cite{fesenko}, III, (1.4)) in the ramification filtration of ${\rm Gal}({\rm L/K})$. Hesselholt proved the following theorem, which proves his conjecture for a large class of extensions.

\begin{thm}[\cite{lh}]\label{hessel}
Let $L/K$ be a totally ramified cyclic extension of order $p$ with $s>e_K/(p-1)$. Then $\ilim H^1(G,W_n(\sO_L))$ is zero.
\end{thm}

\noindent The proabelian group $\ilim H^1(G,W_n(\sO_L))$ can be identified with $H^1(G,W(\sO_L))$ (see Corollary  \ref{limit}). The main result of this paper is, 

\begin{thm} \label{main} Let $L/K$ be a finite Galois extension of complete discrete valued fields with Galois group $G$. Then $ H^1(G,W(\sO_L))$ is zero.
\end{thm}

Although the proof of Theorem \ref{hessel} does not generalize (for instance due to use of (\cite{lh}, $2.2$)), our proof, which is based on an observation on addition in Witt rings (see Lemma \ref{imp2}) relies on several ideas developed in \cite{lh}. One of these ideas which we use is the following.

\begin{lemma} {\rm(\cite{lh}, 1.1)}\label{imp1} Let $L/K$ be as in Theorem \ref{main}. Let $m \geq 1$ be an integer and suppose that the induced map 
$$ H^1(G,W_{m+n}(\sO_L)) \to  H^1(G,W_{n}(\sO_L))$$ 
\noindent is zero for $n=1$. Then the same is true, for all $n \geq 1$. In particular $$\ilim H^1(G,W_n(\sO_L))=0$$ 
\end{lemma} 
\noindent Thus, because of Lemmas \ref{rd}, \ref{imp1} and Corollary \ref{limit}, to prove Theorem \ref{main}, it is enough to prove the following.

\begin{thm} \label{reduction} Let $K$ be as above and ${\rm L/K}$ be degree-$p$ totally ramified cyclic extension with Galois group $G$. Then there exists a positive integer $m \in \N$ such that the homomorphism $H^1(G, W_m(\sO_L))  \to  H^1(G,\sO_L) $ is equal to zero.
\end{thm}

\noindent {\bf Acknowledgement}: Our sincere thanks to Prof. C. S. Rajan for his much help and useful discussions. The second author would also like to thank Joel Riou for his interest and useful comments. 

\section{Preliminaries}

In this section we show that $\ilim H^1(G,W_n(\sO_L))$ coincides with $H^1(G,W(\sO_L))$ (see Corollary \ref{limit}). Note that in general group cohomology does not commute with inverse limits. 

\begin{prop}\label{limprop}
Let $G$ be a finite group and $\{A_i\}_{i\in \N}$ be an inverse system of $G$ modules indexed by $\N$. For $j>i$, let $\phi_{ji}:A_j\to A_i$ denote the given maps. Then the following two statements hold. 
\begin{enumerate}
\item[(i)] If $\phi_{ji}$ is surjective for all $j>i$ then $$ H^1(G,\ilim A_i)\to \ilim H^1(G,A_i)$$ is surjective.
\item[(ii)] If the induced maps $\phi_{ji}^G:A_j^G\to A_i^G$ are surjective for all $j>i$, then 
$$ H^1(G,\ilim A_i)\to \ilim H^1(G,A_i)$$
is injective.
\end{enumerate}
\end{prop}

\begin{cor}\label{limit}
Let $L/K$ be a finite Galois extension of complete discrete valued fields. Then the natural map 
$$\Phi: H^1(G,W(\sO_L))\to \ilim H^1(G,W_n(\sO_L))$$ is an isomorphism.
\end{cor}
\begin{proof}
By construction of Witt vectors, the projection maps $$W_{n+1}(\sO_L) \to W_n(\sO_L)$$ are surjective. Thus by the above proposition, $\Phi$ is surjective. In order to prove injectivity of $\Phi$ we need to prove surjectivity of $$W_{n+1}(\sO_L)^G \to W_n(\sO_L)^G$$ 
This follows from the fact that $W_i(\sO_L)^G=W_i(\sO_K)$ for all $i$ and from the surjectivity of the projection maps $W_{n+1}(\sO_K)\to W_n(\sO_K)$.
\end{proof}

\begin{proof}[Proof of Proposition \ref{limprop}] $(i)$ Suppose we are given an element $\alpha \in \ilim H^1(G,A_i)$. This is equivalent to given $\alpha_i\in H^1(G,A_i)$ for all $i$ such that $\alpha_{i+1}\mapsto \alpha_i$. We now inductively construct cocycles $a^i_g$ representing the class $\alpha_i$ as follows. For $i=1$, choose $a^1_g$ arbitrarily. Now, suppose $a^n_g$ has been constructed. Then construct $a^{n+1}_g$ as follows. First start with any cocycle $b^{n+1}_g$ which represents $\alpha_{n+1}$. For an element $b\in A_{n+1}$, let $\overline{b}$ denote its image in $A_n$. Thus $\overline{b}^{n+1}_g$ is a cocycle in $A_n$ which represents the same class as that represented by $a^n_g$. Thus, there exists $c\in A_n$ such that 
$$ \overline{b}^{n+1}_g-a^n_g = gc-c$$
Since by assumption, $A_{n+1}\to A_n$ is surjective, there exists an element $d\in A_{n+1}$ such that $\overline{d}=c$. Now define 
$$ a^{n+1}_g = b^{n+1}_g-(gd-d)$$
This completes the inductive construction of the cocyles $a^i_g$. The cocyles have the property that for all $i$ and $g$, $$ a^{i+1}_g \mapsto a^i_g$$ and thus they define a cocyles with values in $\ilim A_i$ whose class obviously maps to the element $\alpha$ we started with. \\

\noindent $(ii)$ Suppose $\alpha$ is a class in $H^1(G,\ilim A_i)$ which maps to zero in $\ilim H^1(G,A_i)$, or equivalently maps to zero in $H^1(G,A_i)$ for each $i$. Under the given assumption we will show that $\alpha=0$. Choose a cocyle $a_g$ representing $\alpha$. By abuse of notation, we will denote the image of $a_g$ in $A_n$ by $\overline{a}_g$. The $n$ will be clear from context. 
 For each $n$, we will now inductively construct an element $b_n\in A_n$ such that $$a_g= gb_n-b_n \ \ \ \forall \ g\in G$$ and for all $n$, $b_{n+1}$ maps to $b_n$. For $n=1$, we know that the image of $a_g$ in $A_1$ is a coboundary. Thus there exists an element $b_1\in A_1$ such that 
$$ \overline{a}_g = gb_1-b_1 \ \ \ \forall \ g\in G$$
Now suppose we have defined $b_n$. To define $b_{n+1}$ we first choose an element $c_{n+1}\in A_{n+1}$ such that 
$$ \overline{a}_g = gc_{n+1}-c_{n+1}  \ \ \ \forall \  g\in G$$
However the image of $c_{n+1}$ in $A_n$, denoted by $\overline{c}_{n+1}$ satisfies
$$ g\overline{c}_{n+1}-\overline{c}_{n+1} = gb_n-b_n$$
which means, there exists a $d\in A_n^G$ such that 
$$ b_n=\overline{c}_{n+1}+d$$
Since the map $A_{n+1}^G\to A_n^G$ is assumed to be surjective, we can lift $d$ to an element $\tilde{d}\in A_{n+1}^G$. Now define 
$$ b_{n+1} = c_{n+1}+\tilde{d}$$
The $b_n'$ defined above are compatible elements and hence define an element $b$ of $\ilim A_i$. Also, from the construction it is clear that 
$$ a_g = gb-b \ \ \ \forall \ g\in G$$
holds, since it holds after taking image in $A_i$ for all $i$. Thus the cocycle $a_g$ is actually a coboundary and hence the class $\alpha$ we started with is trivial. 
\end{proof}

\section{Remarks on addition of Witt vectors}

The main observation of this section is Lemma \ref{imp2}, which lies at the heart of the proof of Theorem \ref{main}. We first recall from (\cite{serre},II) how addition of Witt vectors is defined. For every positive integer $n$, define polynomials 
$w_n\in \Z[X_1,...,X_n]$ by 
$$ w_n(X_1,...,X_n) = X_1^{p^{n-1}}+pX_2^{p^{n-2}}+p^2X_3^{p^{n-3}} + \cdots + p^{n-1}X_n $$
One now defines addition of Witt vectors (thanks to Theorem \ref{Witt}) in such a way that if 
\begin{equation}\label{eqn1} (X_1,...,X_n) + (Y_1,...,Y_n) = (Z_1,...,Z_n) \end{equation}
then 
$$ w_i(X_1,...,X_i)+w_i(Y_1,...,Y_i)=w_i(Z_1,...,Z_i) \ \ \forall \ i \leq n$$

\begin{thm} {\rm (\cite{serre}, II.\S6 )} \label{Witt} For every positive integer $n$, there exists a unique $\phi_n \in \Z[X_1,..,X_n,Y_1,..,Y_n]$ such that 
$$ w_n(X_1,...,X_n)+w_n(Y_1,...,Y_n)= w_n(\phi_1,...,\phi_n) \ \ \forall \ n\in \N$$
\end{thm}

In other words, in equation \eqref{eqn1} above 
$$ Z_i = \phi_i(X_1,...,X_i,Y_1,..,Y_i)$$
Note that since $\phi_i's$ are polynomials with integral coefficients, the expression makes sense in all characteristics.

We now consider addition of $p$ Witt vectors. Let 
$$ (x_{11},...,x_{1n})+...+(x_{p1},...,x_{pn}) = (z_1,...,z_n) $$
By above discussion, for every $i\leq n$, there exist polynomials in $ni$ variables, $g_i\in \Z[X_{11},..,X_{1i},...X_{p1},...,X_{pi}]$ such that 
$$z_i = g_i(x_{11},...,x_{1i},...,x_{p1},...,x_{pi})$$
The following observation is about the nature of these polynomials.
\begin{lemma} \label{imp2} Let $R$ be a ring, $p$ be a prime, $n \in \N$ and $W_n(R)$ be the ring of Witt vectors of length $n$. Let $\underline{x}_i = \left(x_{i1},x_{i2},\cdots,x_{in}\right) \in W_n(R)$ for $1 \leq i \leq p$. Let
 $$\left(z_1,z_2,\cdots,z_n\right) := \sum_{i=1}^{p}(x_{i1},x_{i2},\cdots,x_{in} )$$
\begin{enumerate}
\item For all $1 \leq \ell \leq n$ there exists a polynomial expression $f_\ell \in \Z \left[\{x_{ij} \}\right]$ where $1\leq i \leq p, 1 \leq j \leq \ell-1$ such that $$z_{\ell} = \sum_{i=1}^{p}{x_{i\ell}} + f_{\ell} $$
\noindent where each monomial of $f_l$ has degree  $\geq p$. 
\item There exists a polynomial expression $h_{\ell-2} \in \Z\left[\{ x_{ij} \}\right]$ where $1 \leq i \leq p, 1 \leq j \leq {\ell-2}$ such that
 $$ f_{\ell} = \frac{\sum_{i= 1}^{p}x_{i,{\ell-1}}^p - (\sum_{i=1}^{p} x_{i,{\ell-1}})^p }{p} + \frac{1}{p} \sum_{j=1}^{p-1}\binom{p}{j} \sum_{i-1}^{p}x_{i,{\ell-1}}^{p-j}f_{\ell-1}^j + h_{\ell-2}$$
and each monomial appearing in $h_{\ell-2}$ has degree $\geq p^2$.
\end{enumerate}
\end{lemma}

\begin{proof} 
$(1)$ By definition of addition of Witt vectors in Witt ring we have 
$$ \sum_{i=1}^p w_{\ell}(x_{i1},...,x_{p\ell}) = w_{\ell}(z_1,...,z_\ell)$$
Using the expression for the polynomials $w_{\ell}$ and rearranging, we get
$$ z_{\ell} =  \sum_{i=1}^{p}{x_{i\ell}} + f_{\ell} $$
where 
$$ f_{\ell} = \frac{1}{p^{\ell-1}}\left(\sum_{i=1}^p x_{i1}^{p^{\ell-1}}-z_1^{p^{\ell-1}}\right)+...+\frac{1}{p}\left(\sum_{i=1}^px_{i(\ell-1)}^p-z_{\ell-1}^p\right) $$

\noindent The claim that $f_{\ell}$ has integral coefficients follows from Theorem \ref{Witt}. Note that each $z_t$, $t\leq \ell$ in the above expression is again a polynomial expression in the variables $x_{ij}'s, j\leq t$. It is straightforward to observe from the expression of $f_{\ell}$ that every monomial appearing in the expression has degree $\geq p$.

\noindent $(2)$  Substitute $z_{\ell-1}=\sum_{j=1}^px_{j(\ell-1)}+f_{\ell-1}$ in the expression of $f_{\ell}$ and rewrite $f_{\ell}$ as 
$$f_{\ell} = \frac{\left(\sum_{i=1}^{p}x_{i(\ell-1)}^p\right)- \left(\sum_{i=1}^{p}x_{i(\ell-1)}\right)^p }{p}- \frac{1}{p}\sum_{j=2}^{p-1}\binom{p}{j}\left( \sum_{i=1}^p x_{i(\ell-1)}\right) ^{p-j}\cdot f_{\ell-1}^j  + h_{\ell-2}$$
\noindent where \\

$
\begin{array}{ll} 
h_{\ell-2} =  & -\frac{1}{p}f_{\ell-1}^p + \frac{1}{p^2}\left(x_{1(\ell-2)}^{p^2}+x_{2(\ell-2)}^{p^2}+\cdots + x_{p(\ell-2)}^{p^2}-z_{\ell-2}^{p^2} \right)       + \cdots \\
 & + \frac{1}{p^{\ell-1}}\left(x_{11}^{p^{\ell-1}}+x_{21}^{p^{\ell-1}}+ \cdots+x_{p1}^{p^{\ell-1}}-z_1^{p^{\ell-1}}\right)
\end{array} $

\noindent Note that since $p$ is a prime number, every binomial coefficient $\binom{p}{j}$ with $1\leq j<p$, is divisible by $p$. Thus the first two terms in the above expressions of $f_{\ell}$ have integral coefficients. Since we know that $f_{\ell}$ has integral coefficients, it follows that $h_{\ell-2}$ has integral coefficients too. Moreover, since all monomials appearing in $f_{\ell-1}$ have degree $\geq p$, all monomials appearing in $f_{\ell-1}^p$ have degree $\geq p^2$. It is also clear that for $1\leq i\leq \ell-2$, all monomials in 
$$ \frac{1}{p^{\ell-i}}\left(x_{1i}^{p^{\ell-i}}+x_{2i}^{p^{\ell-i}}+ \cdots+x_{pi}^{p^{\ell-i}}-z_i^{p^{\ell-i}}\right)$$
have degree $\geq p^2$. This shows that all monomials appearing in the expression of $h_{\ell-2}$ have degree $\geq p^2$.
\end{proof}

\section{Proof of the main theorem}
We will prove Theorem \ref{main} in this section. 

\begin{lem}\label{rd}
Let $p$ be a prime number and $L/K$ be a finite Galois extension of complete discrete fields with $G=\Gal(L/K)$. Suppose that $k_L/k_K$ is separable. Then the following two statements are equivalent. 
\begin{enumerate}
\item[(i)] $H^1(G,W(\sO_L))=0$ for all extensions $L/K$ as above.
\item[(ii)] $H^1(G,W(\sO_L))=0$ for all $L/K$ as above which are ramified and of degree $p$. 
\end{enumerate}
\end{lem}
\begin{proof}
$(i)\implies (ii)$ is obvious. Now we prove $(i)$ assuming $(ii)$.

Let $L/K$ be any Galois extension of complete discrete valued fields. Let $L^t$ be the maximal subfield of $L$ which is tamely ramified over $K$. The extension $L^t/K$ is Galois and let $H=\Gal(L/L^t)$. Since $L^t/K$ is tame, $\sO_{L^t}$ is a projective $\sO_K[G/H]$ module (see \cite{frohlich}, I. Theorem(3)) which can be used to show the vanishing of $H^1(G/H,W(\sO_{L^t}))$. Moreover, because of the following inflation-restriction exact sequence 
$$ 0 \to H^1(G/H,W(\sO_{L^t})) \stackrel{inf}{\longrightarrow} H^1(G,W(\sO_L)) \stackrel{res}{\longrightarrow} H^1(H,W(\sO_L))$$
vanishing of $H^1(G,W(\sO_L))$ is implied by that of $H^1(H,W(\sO_L))$. Thus without loss of generality, we may replace $K$ by $L^t$ and assume that our extension $L/K$ is totally wildly ramified Galois extension. Thus $G$ is a $p$-group. Since any $p$-group has a normal subgroup of index $p$, again by induction and inflation-restriction exact sequence, we reduce ourselves to the case when $L/K$ is of degree $p$. But in this case the vanishing of $H^1(G,W(\sO_L))$ is guaranteed by $(ii)$. This proves the lemma.  
\end{proof}

Let $G$ be any finite cyclic group with a generator $\sigma$. Let $M$ a $G$-module. Then the cohomology group $H^i(G,M)$ is isomorphic to the $i^{th}$  cohomology group of the complex $$ M \stackrel{1-\sigma}{\longrightarrow} M \stackrel{tr}{\longrightarrow}  M \stackrel{1-\sigma}{\longrightarrow} M \stackrel{tr}{\longrightarrow}  M \to \cdots$$
\noindent where for $a \in  M$, $tr(a)=\sum_{g\in G}ga$. Thus in the case at hand, where $L/K$ is a cyclic Galois extension, we have a canonical isomorphism $$ H^1(G, W_m(\sO_L)) \cong  W_m(\sO_L)^{tr=0}/(\sigma-1) W_m(\sO_L) $$ \\
\noindent Henceforth, for $K$ as before, we assume $L/K$ is a totally ramified cyclic extension of degree $p$. For such an extension we will denote by $s$ the ramification break. To prove the theorem \ref{reduction} we need following lemmas and results from \cite{lh}. 

\begin{lemma} {\rm(\cite{lh}, 2.4)} \label{imp} Let ${\rm L/K}$ be as above. Suppose that $x \in \sO_L^{tr=0}$ represents a non-zero class in ${\rm H^1(G,\sO_L)}$. Then $v_L(x_1) \leq s-1$.
\end{lemma}

\begin{lemma}{\rm (\cite{lh}, 2.1) } \label{vktr} Let ${\rm L/K}$ be as above. For all $a \in \sO_L$, $$v_K(tr(a)) \geq (v_L(a)+s(p-1))/p$$
\end{lemma}

\begin{lemma} {\rm (\cite{lh}, 2.2.)}\label{vksub} Let ${\rm L/K}$ be as above. For all $ a \in \sO_L $, $$v_K(tr(a^p)-tr(a)^p) = e_K + v_L(a)$$
\end{lemma} 

\begin{lemma} \label{C} Let $\underline{x} = (x_1,x_2,\cdots,x_n) \in {W_n(\sO_L)}^{tr=0}$ then  for all $1 \leq \ell \leq n$ $$ - tr(x_\ell) = \frac{tr({x_{\ell-1}^p}) - tr(x_{\ell-1})^p}{p} - C.tr(x_{\ell-1})^p + h_{\ell-2}$$ where $C$ is the integer defined by 
$$ C=\frac{1}{p}\sum_{j=1}^{p-1}(-1)^j\binom{p}{j}$$ and $h_{\ell-2}$ is a polynomial expression in $(x_1,\cdots,x_{\ell-2})$ and it's all ${p-1}$ conjugates. Further each monomial of $h_{\ell-2}$ is of degree $\geq p^2$.
\end{lemma}

\begin{proof} Since $\underline{x}\in W_n(\sO_L)^{tr=0}$ we have 
$$\sum_{i=1}^p (\sigma^{i-1}x_1,...,\sigma^{i-1}x_n) = (0,...0)$$
Thus the above claim follows directly from Lemma \ref{imp2}(2) by making the substitutions 
$$ x_{ij} = \sigma^{i-1}x_j \ \ \ 1 \leq i \leq p, \ \ 1\leq j\leq n,$$
and 
$$ z_i = 0 \ \ \forall \ 1 \leq i \leq n$$
\end{proof}

\begin{lemma} \label{invariant} For $\ell \geq 2$, $h_{\ell-2} \in \sO_K$. Further 
$$ v_K(h_{\ell-2}) \geq p\cdot {\rm min}\{v_L(x_i) \ | \ 1\leq i \leq \ell-2\}$$
\end{lemma}
\begin{proof} The polynomial expression for $h_{\ell-2}$ in $x_1,...,x_n$ and its conjugates can be seen to be invariant under the Galois action. Hence it belongs to $\sO_K$. Further since $h_{\ell-2}$ is a sum of monomials in $x_i's, i\leq \ell-2$ and its conjugates, each of degree $\geq p^2$, we have 
$$ v_L(h_{\ell-2}) \geq p^2\cdot  {\rm min}\{v_L(x_i) \ | \ 1\leq i \leq \ell-2\}$$
The lemma now follows from the fact that $v_L(h_{\ell-2})=p\cdot v_K(h_{\ell-2})$.
\end{proof}

\begin{proof}[Proof of Theorem \ref{reduction}]By the Lemma \ref{imp}, to prove the Theorem \ref{reduction} it is sufficient to find $M \in \N$ such that, for all $\underline{x} =(x_1,...,x_M) \in  W_M(\sO_L)^{tr=0}$, $v_L(x_1) \geq s$. \\

\noindent \underline{Step(1)}: Let $n$ be a positive integer and $(x_1,...,x_n)\in W_n(\sO_L)^{tr=0}$. We will prove by induction on $\ell$ that $v_L(x_{\ell}) \geq \frac {s(p-1)}{p}$ for $ 1 \leq \ell \leq n-1$. \\

\noindent By Lemma \ref{C}, and using the fact that $h_0=0$, $tr(x_1)=0$ we have 
$$ -tr(x_2) = \frac{1}{p}(tr(x_1^p)-tr(x_1)^p)$$
But by Lemma \ref{vktr}, $v_K(tr(x_2)) \geq \frac{s(p-1)}{p}$. Thus
$$    v_K(tr(x_1^p)-tr(x_1)^p)-e_K = v_K(tr(x_2))\geq \frac{s(p-1)}{p}$$
By Lemma \ref{vksub} $ v_K(tr(x_1^p)-tr(x_1)^p)=v_L(x_1)+e_K$. Therefore $v_L(x_1)\geq \frac{s(p-1)}{p}$. This proves the claim for $\ell =1$. \\

\noindent Now assume that for all $i\leq \ell-1$, $v_L(x_i)\geq \frac{s(p-1)}{p}$. We will prove $v_L(x_{\ell})\geq \frac{s(p-1)}{p}$. By Lemma \ref{C}, we have 
$$  - tr(x_{\ell+1}) = \frac{tr(x_\ell^p) - tr(x_\ell)^p}{p} - C\cdot tr(x_\ell)^p + h_{\ell-1} $$
Thus, using Lemma \ref{vksub} we get 
$$ v_L(x_{\ell}) = v_K(\frac{tr(x_\ell^p) - tr(x_\ell)^p}{p}) \geq {\rm inf}\{v_K(tr(x_{\ell+1})),v_K(C\cdot tr(x_{\ell})^p),v_K(h_{\ell-1})\}$$
Using Lemma \ref{vktr}, we have $$v_K(tr(x_{\ell+1}))\geq s(p-1)/p \ \ \ \text{and}$$ $$v_K(C\cdot tr(x_{\ell})^p)\geq s(p-1)$$ By Lemma \ref{invariant}, and by induction hypothesis $v_K(h_{\ell-1})\geq s(p-1)$. Combining the above, we get 
$$v_L(x_{\ell})\geq \frac{s(p-1)}{p}$$

\noindent \underline{Step(2)}: Now we will show existance of $M \in \N$ such that for all $\underline{x} \in W_M(\sO_L)$, $v_L(x_1) \geq s$. For any positive integer $n$ and $(x_1,...,x_n)\in W_n(\sO_L)^{tr=0}$, by Step(1) we have 
$$v_L(x_i) \geq \frac{s(p-1)}{p}, \ \ \ \forall \   1 \leq i \leq {n-1}.$$

\noindent For a fixed $n$, and $1\leq i\leq n-1$, we claim that 
$$v_L(x_{n-i}) \geq \frac{s(p-1)}{p}(1+\frac{1}{p}+\cdots +\frac{1}{p^{i-1}})$$
We prove this by induction on $i$. For $i=1$, this is the claim that 
$$v_L(x_{n-1}) \geq \frac{s(p-1)}{p}$$ which follows from Step(1). Now assuming the claim for a general $1 \leq i \leq n-2$, we will prove it for $i+1$.  By induction hypothesis $$v_L(x_{n-i})\geq \frac{s(p-1)}{p}(1+\frac{1}{p}+ \cdots + \frac{1}{p^{i-1}})$$
Therefore by using Lemma \ref{vktr} we get 
$$v_K(tr(x_{n-i})) \geq \frac{v_L(x_{n-i})+s(p-1)}{p} \geq \frac{s(p-1)}{p}(1+\frac{1}{p}+ \cdots + \frac{1}{p^{i}})$$

By Lemma \ref{C}
$$ -tr(x_{n-i}) = \frac{tr(x_{n-(i+1)}^p) - tr(x_{n-(i+1)})^p}{p} - C\cdot tr(x_{n-(i+1)})^p + h_{n-(i+2)}$$
By Lemma \ref{vktr}, $v_K(C\cdot tr(x_{n-(i+1)})^p)\geq s(p-1)$. By Step(1) and Lemma \ref{invariant}, $v_K(h_{n-(i+2)})\geq s(p-1)$. Thus, using Lemma \ref{vksub},\\

$
\begin{array}{ll}
 v_L(x_{n-(i+1)}) &= v_K\left(\frac{tr(x_{n-(i+1)}^p) - tr(x_{n-(i+1)})^p}{p}\right) \\
                &\geq {\rm min}\{ v_K(tr(x_{n-i})),v_K(C\cdot tr(x_{n-(i+1)})^p),v_K(h_{n-(i+2)}) \} \\
                & \geq {\rm min}\{ \frac{s(p-1)}{p}(1+\frac{1}{p}+ \cdots + \frac{1}{p^{i}}),s(p-1),s(p-1)\}\\
                &= \frac{s(p-1)}{p}(1+\frac{1}{p}+ \cdots + \frac{1}{p^{i}})
\end{array}
$\\
\noindent This proves the claim. Hence 
$$ v_L(x_1) \geq \frac{s(p-1)}{p}(1+\frac{1}{p}+ \cdots + \frac{1}{p^{n-2}})$$

We know that as $n\to \infty$, $\frac{s(p-1)}{p}(1+\frac{1}{p}+ \cdots + \frac{1}{p^{n-2}})\to s$. There exists an integer $M$, such that 
$$ \frac{s(p-1)}{p}(1+\frac{1}{p}+ \cdots + \frac{1}{p^{M-2}}) >s-1$$ Since $v_L$ is a discrete valuation, for such $M$ and for any 
$(x_1,...,x_M)\in W_M(\sO_L)^{tr=0}$, we have shown that 
$$ v_L(x_1) \geq s$$

\end{proof}


\begin{thebibliography}{xx}

\bibitem{fesenko} Fesenko, I. B., Vostokov, S. V.; {\it Local fields and their extensions}. Second edition. Translations of Mathematical Monographs, {\bf 121}. American Mathematical Society, Providence, RI, 2002.

\bibitem{frohlich} Fr$\rm \ddot{o}$hlich, Albrecht; {\it Galois module structure of algebraic integers}. {\bf 3}. Springer-Verlag, Berlin, 1983.

\bibitem{lh} Lars Hesselholt;  Galois cohomology of Witt vectors of algebraic integers.
{\it Math. Proc. Cambridge Philos. Soc.} {\bf 137} (2004), no. {\bf 3}, 551–557. 

\bibitem{serre} Jean-Pierre Serre; {\it Local fields}. Translated from the French by Marvin Jay Greenberg. GTM {\bf 67}. Springer-Verlag, New York-Berlin, 1979.

\end{thebibliography}
\end{document}